\documentclass[12pt]{article}
\usepackage{amsmath,amsfonts,amssymb, amsthm,enumerate,graphicx}
\usepackage{tikz,authblk}
\usepackage{subfig}

\theoremstyle{plain}
\newtheorem{theorem}{{\bf Theorem}}[section]
\newtheorem{lemma}[theorem]{{\bf Lemma}}
\newtheorem{corollary}[theorem]{{\bf Corollary}}
\newtheorem{proposition}[theorem]{{\bf Proposition}}

\theoremstyle{definition}
\newtheorem{remark}[theorem]{{\bf Remark}}
\newtheorem{definition}[theorem]{{\bf Definition}}

\newcommand{\lk}[2]{{\rm lk}_{#1}(#2)}

\newcommand{\skel}[2]{{\rm skel}_{#1}(#2)}
\newcommand{\Kd}[1]{{\mathcal K}(#1)}
\newcommand{\bs}{\backslash}

\newcommand{\meets}{\leftrightarrow}
\newcommand{\nmeets}{\nleftrightarrow}

\author[1] {Benjamin A.~Burton}
\author[2] {Basudeb Datta}
\author[3] {Nitin Singh}
\author[4] {Jonathan Spreer}

\affil[1,4] {School of Mathematics and Physics, The University of Queensland, Brisbane, QLD 4072, Australia. bab@maths.uq.edu.au, j.spreer@uq.edu.au.}

\affil[2] {Department of Mathematics, Indian Institute of Science,
Bangalore 560\,012, India. dattab@math.iisc.ernet.in. }

\affil[3] {IBM Research India Lab, Manyata Embassy Business Park,  
8th Floor, G2, Bangalore  560\,045, India. nitsingk@in.ibm.com.}

\title{Separation index of graphs and stacked 2-spheres}

\date{To appear in {\bf Journal of Combinatorial Theory A}}

\begin{document}

\maketitle


\begin{abstract}

In 1987, Kalai proved that stacked spheres of dimension $d\geq 3$ are 
characterised by the fact that they attain equality in Barnette's celebrated Lower Bound Theorem. This result does not extend to dimension $d=2$. In this article, we give a characterisation 
of stacked $2$-spheres using what we call the {\em separation index}. Namely, we show that the separation index of a triangulated $2$-sphere is maximal if and only if it is stacked. In addition, we prove that, amongst all $n$-vertex triangulated $2$-spheres, the separation index is {\em minimised} by some $n$-vertex flag sphere for $n\geq 6$.

Furthermore, we apply this characterisation of stacked $2$-spheres to settle the outstanding $3$-dimensional case of the Lutz-Sulanke-Swartz conjecture that ``tight-neighbourly triangulated manifolds are tight''. 
For dimension $d\geq 4$, the conjecture has already been proved by Effenberger following a result of Novik and Swartz.

\end{abstract}

\noindent {\small {\em MSC 2010\,:} 57Q15, 57M20, 05C40.

}

\section{Introduction and results}

For a graph $G$, we investigate its measure of ``average"
disconnectivity $s(G)$, which we call its {\em separation index}. 
Roughly, $s(G)$ is the weighted average of the number of connected components over all induced subgraphs of $G$. This measure 
already appears in Hochster's formula \cite{Hochster77}, which relates the reduced homology groups of a simplicial complex to the graded Betti numbers of its associated Stanley-Reisner ideal. 
It also features as part of the sigma- and mu-vectors for simplicial
complexes introduced by Bagchi and Datta \cite{BDEJC} 
for studying tight triangulations (see Definition
\ref{def:sigmamu} below). 

While one may study $s(G)$ for graphs in general, in this paper we consider the case when $G$ is the $1$-skeleton of a triangulated 
$2$-sphere. For a simplicial complex $X$, $s(X)$ will denote the
separation index of the $1$-skeleton of $X$. We show: 

\begin{theorem}\label{thm:result1}
Let $X$ be an $n$-vertex triangulated $2$-sphere. Then $$s(X)\leq (n 
-8)(n +1)/20,$$ where equality occurs if and only if $X$ is a stacked sphere.
\end{theorem}

Theorem~\ref{thm:result1} provides both an upper bound for the separation index and a characterisation of stacked $2$-spheres. This supplements the 
characterisation of stacked $d$-spheres, $d \geq 3$, as the  triangulations which attain equality in Barnette's Lower Bound Theorem \cite{Ba,KA}. In dimension two all triangulated $n$-vertex $2$-spheres have
equal numbers of faces of all dimensions, and hence the Lower Bound Theorem cannot be used to characterise stacked spheres. 

Intuitively, it seems plausible that stacked spheres maximise the number of connected components of induced subcomplexes (and thus the separation index): stacked spheres naturally contain a large number of separating cycles of minimal length (i.e., separating $3$-cycles).
 
Regarding small separation indices, we prove: 

\begin{theorem}\label{thm:two}
Amongst all $n$-vertex triangulated $2$-spheres, $n \neq 5$, the one with minimum separation index is a flag $2$-sphere.
\end{theorem}

Note that Theorem~\ref{thm:two} is not true for $n = 5$ since there is exactly one $5$-vertex triangulated $2$-sphere which is stacked but not flag. 

Theorem~\ref{thm:two} seems plausible from an intuitive point of view:  flag 
$2$-spheres are opposite to stacked sphere in the sense that they do not contain separating $3$-cycles at all. Hence, fewer induced subcomplexes can be associated with a large number of connected components. Note, however, that in the case of flag $2$-spheres there are usually many distinct separation indices possible for a fixed number of vertices. For instance, there are $87$ flag $2$-spheres with $12$ vertices, and these have $60$ distinct separation indices. So, flag $2$-spheres appear to be more diverse than stacked spheres. 

A striking implication of Theorem \ref{thm:result1} is that all 
tight-neighbourly $3$-manifolds have stacked spheres as vertex-links, i.e., they belong to Walkup's class $\Kd{3}$ of triangulated 3-manifolds. Indeed, we prove:

\begin{theorem}\label{thm:result2}
Let $X$ be a triangulated $3$-manifold. If $X$ is tight-neighbourly then $X$ is a neighbourly member of $\Kd{3}$. 
\end{theorem}

Novik and Swartz \cite{NS} proved a similar result for dimension 
$d\geq 4$. More precisely, they proved that a tight-neighbourly triangulated $d$-manifold belongs to $\Kd{d}$ for $d\geq 4$. 
In \cite{LSS}, Lutz, Sulanke and Swartz conjectured that, for $d\geq 3$, all tight-neighbourly triangulated $d$-manifolds are tight.
Using Novik-Swartz's result, Effenberger \cite{EFF} proved this conjecture for $d\geq 4$. Here, as a consequence of Theorem \ref{thm:result2} and Proposition \ref{prop:bd} below, we prove the 
Lutz-Sulanke-Swartz conjecture in the remaining case $d=3$. That is:

\begin{corollary}\label{cor:tntight}
Let $X$ be a tight-neighbourly triangulated $3$-manifold. Then $X$ is tight.
\end{corollary}

Although the converse of Theorem  \ref{thm:result2} is true in dimensions 
$d\geq 4$ (see Corollary \ref{cor:lss}), it is not true in dimension three 
(see Remark \ref{remark:converse}).

\begin{remark} \label{remark:character}
Tight-neighbourly $3$-manifolds are more common than one might think.
In addition to the $5$-vertex standard $3$-sphere and the $9$-vertex non-sphere 
triangulated $3$-manifold found by Walkup, the authors recently discovered 
$75$ tight-neighbourly $3$-manifolds of five additional topological types 
\cite{BDSS2}. These include the two $29$-vertex examples in \cite{DS}. 
Among these $75$ examples, six are $29$-vertex, one is $49$-vertex, $15$ are 
$69$-vertex, $41$ are $89$-vertex and $12$ are $109$-vertex triangulations. 
The list is by no means expected to be complete. However, the existence of
an infinite family of such triangulations is unknown as of today. This is work 
in progress.

Furthermore, there is no tight triangulated $3$-manifold known which is
{\em not} tight-neighbourly. In fact, such a triangulation, if it exists, must   
fulfil very strong conditions. For instance, following a recent result of 
Bagchi and the second and fourth authors \cite{BDS}, {\em any} tight 
triangulated $3$-manifold with first Betti number smaller than $189$ must be 
tight-neighbourly.
\end{remark}

\section{Preliminaries and basic results}

\subsection{Simplicial complexes and graphs}

All simplicial complexes considered here are finite and abstract. The
vertex set of a simplicial complex $X$ is denoted by $V(X)$. By a
{\em triangulation} of a space $M$, we mean a simplicial complex $X$ whose
geometric carrier is $M$. By a {\em triangulated $d$-manifold} (resp., 
{\em $d$-sphere}) we mean a triangulation of a topological manifold (resp., sphere) of dimension $d$. 
The boundary complex of a $(d+1)$-simplex is a triangulated $d$-manifold with $d+2$ vertices and  
triangulates the $d$-sphere $S^d$. It is called the {\em standard $d$-sphere} and is denoted by $S^d_{d+2}$. 
A simplicial complex of dimension $d$ is called {\em pure} if all its maximal
faces are $d$-dimensional. For a simplicial complex $X$, and $A\subseteq
V(X)$, $X[A]$ denotes the simplicial complex consisting of all faces of
$X$ which are contained in $A$. We say that $X[A]$ is the subcomplex of $X$   
{\em induced} by the set $A$. 

For a finite set $A$, let ${\rm Cl}(A)$ denote the simplicial complex
consisting of all subsets of $A$. The {\em link} of a vertex $x$ in a simplicial
complex $X$ is defined to be the subcomplex $\lk{X}{x} := \{\alpha\in X: x\not\in\alpha,
\alpha\cup \{x\}\in X\}$. For $k \leq \dim(X)$, we define $\skel{k}{X} := \{\alpha\in X:
|\alpha|\leq k+1\}$ to be the {\em $k$-skeleton} of the simplicial complex $X$. 

For a $d$-dimensional simplicial complex $X$, the
vector $(f_0,f_1,\ldots,f_d)$ is called its $f$-{\em vector}, where
$f_i=f_i(X)$ is the number of $i$-dimensional faces of $X$. We will call a
simplicial complex {\em neighbourly} if $f_1=\binom{f_0}{2}$, i.e., any two
vertices form an edge.

A simplicial complex is called {\em flag}, if any $j$-element subset
of its vertices which spans a clique also spans a $(j-1)$-simplex, $j \geq 2$. 
For instance, a triangulated $2$-sphere is flag if and only if it has no 
induced $3$-cycle (i.e., there is no set of three vertices spanning three
edges but not a triangle). Since a flag 2-sphere is not a connected sum of two triangulated 2-spheres, a flag 2-sphere is also called {\em primitive}  \cite{BDS}. 

Unless the field is explicitly mentioned, the homologies and Betti numbers are considered 
w.r.t.\ $\mathbb{Z}_2$, but the arguments hold for an arbitrary field $\mathbb{F}$,
when the manifold is $\mathbb{F}$-orientable. So, $H_i(X) = H_i(X;\mathbb{Z}_2)$ and $\beta_i(X) = \beta_i(X;\mathbb{Z}_2)$ for 
all $i\geq 0$ and for all simplicial complexes $X$.

All graphs considered here are finite and simple. A standard reference for
basic terminology on graphs is \cite{BM}. For a graph $G$, $V(G)$ and
$E(G)$ will denote its vertex-set and edge-set respectively. For $v\in
V(G)$, $d_G(v)$ denotes the {\em degree} of $v$ in $G$. The set of
neighbours of $v$ in $G$ is denoted by $N_G(v)$, or just $N(v)$ when the
ambient graph is clear from the context. For $n\geq 3$, an $n$-cycle with edges $u_1u_2, \dots, u_{n-1}u_n, u_nu_1$ is denoted by 
$C_n(u_1, u_2, \dots, u_n)$. A graph is called {\em planar} if
it can be embedded in a plane (or $2$-sphere) without the edges
intersecting in an interior point. The following are well known.

\begin{lemma}\label{lem:graphs}
Let $G$ be a planar graph. Then,
\begin{enumerate}[{\rm (a)}]
\item $G$ does not contain $K_5$ as a subgraph; 
\item $|E(G)|\leq 3|V(G)|-6$.
\end{enumerate}
\end{lemma}

\subsection{Stacked spheres} 

Let $X$ be a pure $d$-dimensional simplicial complex and let $x\not\in V(X)$. We say that $Y$ is
obtained from $X$ by {\em starring} the vertex $x$ in the $d$-face $\sigma$ of $X$, if
$Y=(X\bs\{\sigma\})\cup \{\tau\cup \{x\}: \tau \subset \sigma\}$. A simplicial complex is called a {\em stacked $d$-sphere} if it
is obtained from $S^d_{d+2}$ by a finite sequence of starring operations.
It is clear that a stacked $d$-sphere triangulates the $d$-sphere $S^d$. 
We know that a stacked $d$-sphere has at least two vertices of degree $d+1$, i.e., whose links are standard $(d-1)$-spheres with $d+1$ vertices (cf. Lemma 4.3\,(b) in \cite{BDLBT}). 

In \cite{WALKUP}, Walkup defined the class $\Kd{d}$ of triangulated $d$-manifolds whose vertex-links are stacked $(d-1)$-spheres. 

\subsection{Tight-neighbourly and tight triangulated manifolds}

The following result by Novik and Swartz \cite{NS} gives an upper bound on the first Betti
number of a triangulated $d$-manifold depending on its $1$-skeleton. They
prove:

\begin{proposition}[Novik, Swartz]\label{prop:novikswartz}
Let $X$ be a connected triangulated $d$-manifold.
\begin{enumerate}[{\rm (a)}]
\item If $d\geq 3$ then $\binom{d+2}{2}\beta_1(X)\leq
f_1(X)-(d+1)f_0(X)+\binom{d+2}{2}$. 
\item Further, if $d\geq 4$ and 
$\binom{d+2}{2}\beta_1(X)=f_1(X)-(d+1)f_0(X)+\binom{d+2}{2}$ then $X\in \Kd{d}$.
\end{enumerate}
\end{proposition}

In \cite{LSS}, Lutz, Sulanke and Swartz observed that Proposition
\ref{prop:novikswartz} implies:

\begin{corollary}\label{cor:lss}
Let $X$ be a connected triangulated $d$-manifold. If $d\geq 3$, then
\begin{align}\label{eq:tn}
\binom{d+2}{2}\beta_1(X) & \leq \binom{f_0(X)-d-1}{2}.
\end{align}
Moreover for $d\geq 4$, equality holds if and only if $X$ is a
neighbourly member of $\Kd{d}$.
\end{corollary}

For $d\geq 3$, a triangulated $d$-manifold is called {\em tight-neighbourly}
if it satisfies \eqref{eq:tn} with equality. From part (a) of Proposition
\ref{prop:novikswartz} and the trivial inequality $f_1(X)\leq
\binom{f_0(X)}{2}$ it can be seen that tight-neighbourly triangulated
manifolds are neighbourly. 

A $d$-dimensional connected simplicial complex $X$ is said to be {\em tight} if for all
$A\subseteq V(X)$, the homology maps induced by the inclusion map, namely 
$H_i(X[A]; \mathbb{Z}_2)\rightarrow H_i(X; \mathbb{Z}_2)$, are injective for all $0\leq i\leq d$.
Despite many new examples described in recent literature \cite{BDSS2,DS},
examples of tight triangulated manifolds are still considered to be 
extremely rare. Nonetheless, they have so far evaded
complete combinatorial characterisation. 
From \cite{EFF} and \cite{BDEJC}, we know the following:

\begin{proposition}[Effenberger]\label{prop:eff}
For $d\neq 3$, the neighbourly members of $\Kd{d}$ are tight.
\end{proposition}

\begin{proposition}[Bagchi, Datta]\label{prop:bd}
If $X$ is a neighbourly member of $\Kd{3}$, then $X$ is tight if and only if
$X$ is tight-neighbourly. 
\end{proposition}

\subsection{Bistellar flips}

\begin{figure}[htbp]
\centering
\subfloat[$0$ and $2$ moves\label{subfig:move0}]{
\begin{tikzpicture}
\def\sqrtthree{1.73}
\coordinate (A) at (0,0);
\coordinate (B) at (2,0);
\coordinate (C) at (1,\sqrtthree);
\coordinate (X) at (1,0.3*\sqrtthree);
\draw (A) -- (B) -- (C) -- cycle;
\draw (A) node[below left] {$a$};
\draw (B) node[below right] {$b$};
\draw (C) node[above] {$c$};

\begin{scope}[shift={(5,0)}] 
\coordinate (A) at (0,0);
\coordinate (B) at (2,0);
\coordinate (C) at (1,\sqrtthree);
\coordinate (X) at (1,0.3*\sqrtthree);
	\draw (A) -- (B) -- (C) -- cycle;
	\draw (X) -- (A) (X) -- (B) (X) -- (C);
\draw (A) node[below left] {$a$};
\draw (B) node[below right] {$b$};
\draw (C) node[above] {$c$};
\draw (X) node[above left] {\small $x$};
\end{scope}
\draw [thick,->] (2.8,0.6*\sqrtthree) -- ++(1.4,0);
\draw [thick,->] (4.2,0.4*\sqrtthree) -- ++(-1.4,0);
\draw (3.5,0.6*\sqrtthree) node[above] {\small $0$-move};
\draw (3.5,0.4*\sqrtthree) node[below] {\small $2$-move};

\end{tikzpicture}
}

\subfloat[$1$ moves\label{subfig:move1}]{
\begin{tikzpicture}
\def\sqrtthree{1.73}
\coordinate (A) at (0,0.5*\sqrtthree);
\coordinate (B) at (-1*\sqrtthree,0);
\coordinate (C) at (0,-0.5*\sqrtthree);
\coordinate (D) at (1*\sqrtthree,0);
\draw (A) -- (B) -- (C) -- (D) --cycle;
\draw (B) -- (D);
\draw (A) node[above] {$a$};
\draw (B) node[left] {$b$};
\draw (C) node[below] {$c$};
\draw (D) node[right] {$d$};

\begin{scope}[shift={(4.0*\sqrtthree,0)}] 
\coordinate (A) at (0,0.5*\sqrtthree);
\coordinate (B) at (-1*\sqrtthree,0);
\coordinate (C) at (0,-0.5*\sqrtthree);
\coordinate (D) at (1*\sqrtthree,0);
\draw (A) -- (B) -- (C) -- (D) --cycle;
\draw (A) -- (C);
\draw (A) node[above] {$a$};
\draw (B) node[left] {$b$};
\draw (C) node[below] {$c$};
\draw (D) node[right] {$d$};
\end{scope}
\draw[thick,<->] (1.5*\sqrtthree,0) -- (2.5*\sqrtthree,0);
\draw (2.0*\sqrtthree,0) node[above] {\small $1$-moves};
\end{tikzpicture}
}
\caption{Bistellar moves}
\label{fig:pachner}
\end{figure}
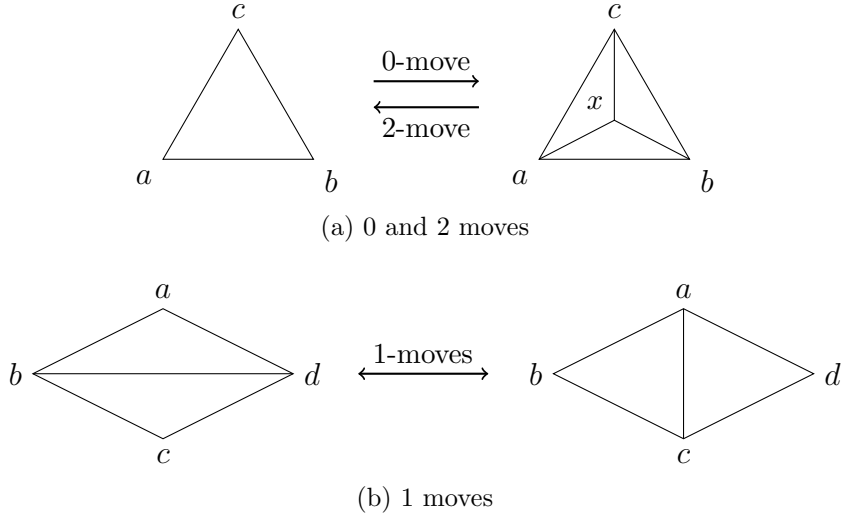

{\em Bistellar flips} or {\em Pachner moves} are ways of replacing a
combinatorial triangulation of a piecewise linear manifold with another
such triangulation of the same manifold. In dimension two, we have the following
bistellar moves:
\begin{enumerate}[{\rm (a)}]
\item ({\bf Bistellar 0-, 2-moves:}) Let $X$ be a two-dimensional pure simplicial complex. If
$Y$ is obtained from $X$ by starring a vertex $x$ in the face $abc$ of
$X$, we say that $Y$ is obtained from $X$ by the bistellar $0$-move
$abc\mapsto x$. We also say that $X$ is obtained from $Y$ by the
bistellar $2$-move $x\mapsto abc$  
(see Figure \ref{fig:pachner}\subref{subfig:move0}).
 
\item ({\bf Bistellar 1-moves:}) Let $X$ be a pure two-dimensional simplicial complex
and let $abd$ and $bdc$ be two adjacent faces of $X$ such that $ac$ is not
an edge of $X$. If $Y=(X\bs\{abd,bdc,bd\})$ $\cup \{abc,acd,ac\}$ then $Y$ and $X$ triangulate the same space. 
We say that $Y$ is obtained from $X$ by the bistellar $1$-move $bd\mapsto ac$. 
Observe that, in this case, $X$ is obtained from $Y$ by the bistellar move $ac\mapsto bd$ (see
Figure \ref{fig:pachner}\subref{subfig:move1}).  
\end{enumerate}

As a consequence of Pachner's classical theorem in  \cite{PACH} we have,

\begin{proposition}[Pachner]\label{thm:pachner}
Any triangulated $2$-sphere can be obtained from the standard $2$-sphere $S^2_4$ by a finite sequence of bistellar $0$-, $1$- and $2$-moves.
\end{proposition}

\begin{definition}[$1A$-move and $1B$-move]\label{def:1a1b}
Let $X$ be a pure two-dimensional simplicial complex. Let $a,b,c,d$ be vertices of 
$X$ such that
$abd$, $bdc$ are two adjacent faces and $ac$ is not an edge. Then the $1$-move $bd \mapsto ac$ is said to be of type $1A$ if one of $a$ and $c$ is a vertex of degree $3$ in $X$. Similarly the 1-move $bd\mapsto ac$ is said to be of
type $1B$ if the degree of one of $a$ and $c$ in $X$ is precisely 4 and the degree of the other one is at least $4$.
\end{definition}

We will need the following slightly stronger version of Proposition \ref{thm:pachner} for
our purpose. In fact, the arguments in the proof of Lemma \ref{lem:pachner} will be used again later in the proof of Theorem \ref{thm:result2}. 

\begin{lemma}\label{lem:pachner}
Let $X$ be a triangulated $2$-sphere. Then $X$ can be obtained from $S^2_4$
by a finite sequence of bistellar $0$-moves, $1A$-moves and $1B$-moves.
\end{lemma}

\begin{proof}
We proceed by induction on the number of vertices in $X$, i.e., on $n=f_0(X)$. The lemma is trivially true for $n=4$. So, assume that $n\geq
5$ and the lemma is true for all triangulated 2-spheres $Y$ with $f_0(Y) <n$. 
Let $f_0(X)=n$. 
Since the $1$-skeleton of $X$ is a planar graph, $X$ must have a vertex of
degree at most $5$. We have the following cases: 

\begin{figure}[htbp]
\centering
\subfloat[$0,1A$ sequence:\label{subfig:01A}]{
\begin{tikzpicture}[yscale=0.8]
\def\side{1.71};
\coordinate (B) at (0,\side);
\coordinate (C) at (-1,0);
\coordinate (D) at (0,-\side);
\coordinate (A) at (1,0);
\coordinate (X) at (0,0.3*\side);

\draw (A) -- (B) -- (C) -- (D) -- cycle;
\draw (A) -- (C);
\draw (A) node[right] {\small $a$};
\draw (B) node[above] {\small $b$};
\draw (C) node[left] {\small $c$};
\draw (D) node[below] {\small $d$};

\begin{scope}[shift={(4,0)}]
\coordinate (B) at (0,\side);
\coordinate (C) at (-1,0);
\coordinate (D) at (0,-\side);
\coordinate (A) at (1,0);
\coordinate (X) at (0,0.3*\side);
\draw (A) -- (B) -- (C) -- (D) -- cycle;
\draw (X) -- (A) (X) -- (B) (X) -- (C);
\draw (A) -- (C);

\draw (A) node[right] {\small $a$};
\draw (B) node[above] {\small $b$};
\draw (C) node[left] {\small $c$};
\draw (D) node[below] {\small $d$};
\draw (X) node[below] {\small $x$};
\end{scope}

\begin{scope}[shift={(8,0)}]
\coordinate (B) at (0,\side);
\coordinate (C) at (-1,0);
\coordinate (D) at (0,-\side);
\coordinate (A) at (1,0);
\coordinate (X) at (0,0.3*\side);
\coordinate (Xd) at (0,0.3*\side-0.1);
\draw (A) -- (B) -- (C) -- (D) -- cycle;
\draw (X) -- (A) (X) -- (B) (X) -- (D) (X) -- (C);

\draw (A) node[right] {\small $a$};
\draw (B) node[above] {\small $b$};
\draw (C) node[left] {\small $c$};
\draw (D) node[below] {\small $d$};
\draw (Xd) node[below left] {\small $x$};
\end{scope}

\draw[thick,->] (1+0.5,0.6*\side) -- (3-0.5,0.6*\side);
\draw (2.0,0.6*\side) node[above] {\scriptsize $(0)$};
\draw[thick,->] (5+0.5,-0.6*\side) -- (7-0.5,-0.6*\side);
\draw (6.0,-0.6*\side) node[below] {\scriptsize $(1A)$};
\end{tikzpicture}
}

\subfloat[$0,1A,1B$ sequence:\label{subfig:01A1B}]{
\begin{tikzpicture}
\def\radius{1.2};
\def\angle{72};

\path (0,0) -- ++(90:\radius) coordinate (A);
\path (0,0) -- ++(90+\angle:\radius) coordinate (B);
\path (0,0) -- ++(90+2*\angle:\radius) coordinate (C);
\path (0,0) -- ++(90+3*\angle:\radius) coordinate (D);
\path (0,0) -- ++(90+4*\angle:\radius) coordinate (E);

\coordinate (X) at (0,0);
\coordinate (Xd) at (0,-0.1);

\draw (A) -- (B) -- (C) -- (D) -- (E) -- cycle;
\draw (A) -- (C) (A) -- (D);

\draw (A) node[above] {\small $a$};
\draw (B) node[left] {\small $b$};
\draw (C) node[below] {\small $c$};
\draw (D) node[right] {\small $d$};
\draw (E) node[right] {\small $e$};

\begin{scope}[shift={(3.5,0)}]
\path (0,0) -- ++(90:\radius) coordinate (A);
\path (0,0) -- ++(90+\angle:\radius) coordinate (B);
\path (0,0) -- ++(90+2*\angle:\radius) coordinate (C);
\path (0,0) -- ++(90+3*\angle:\radius) coordinate (D);
\path (0,0) -- ++(90+4*\angle:\radius) coordinate (E);

\coordinate (X) at (0,0);
\coordinate (Xd) at (0,-0.1);
\draw (A) -- (B) -- (C) -- (D) -- (E) -- cycle;
\draw (A) -- (C) (A) -- (D);
\draw (X) -- (A) (X) -- (C) (X) -- (D);

\draw (A) node[above] {\small $a$};
\draw (B) node[left] {\small $b$};
\draw (C) node[below] {\small $c$};
\draw (D) node[right] {\small $d$};
\draw (E) node[right] {\small $e$};
\draw (Xd) node[below] {\small $x$};
\end{scope}

\begin{scope}[shift={(7,0)}]
\path (0,0) -- ++(90:\radius) coordinate (A);
\path (0,0) -- ++(90+\angle:\radius) coordinate (B);
\path (0,0) -- ++(90+2*\angle:\radius) coordinate (C);
\path (0,0) -- ++(90+3*\angle:\radius) coordinate (D);
\path (0,0) -- ++(90+4*\angle:\radius) coordinate (E);

\coordinate (X) at (0,0);
\coordinate (Xd) at (0,-0.1);
\draw (A) -- (B) -- (C) -- (D) -- (E) -- cycle;
\draw (X) -- (A) (X) -- (C) (X) -- (D);
\draw (X) -- (B) (A) -- (D);

\draw (A) node[above] {\small $a$};
\draw (B) node[left] {\small $b$};
\draw (C) node[below] {\small $c$};
\draw (D) node[right] {\small $d$};
\draw (E) node[right] {\small $e$};
\draw (Xd) node[below] {\small $x$};
\end{scope}

\begin{scope}[shift={(10.5,0)}]
\path (0,0) -- ++(90:\radius) coordinate (A);
\path (0,0) -- ++(90+\angle:\radius) coordinate (B);
\path (0,0) -- ++(90+2*\angle:\radius) coordinate (C);
\path (0,0) -- ++(90+3*\angle:\radius) coordinate (D);
\path (0,0) -- ++(90+4*\angle:\radius) coordinate (E);

\coordinate (X) at (0,0);
\coordinate (Xd) at (0,-0.1);

\draw (A) -- (B) -- (C) -- (D) -- (E) -- cycle;
\draw (X) -- (A) (X) -- (C) (X) -- (D);
\draw (X) -- (B) (X) -- (E);

\draw (A) node[above] {\small $a$};
\draw (B) node[left] {\small $b$};
\draw (C) node[below] {\small $c$};
\draw (D) node[right] {\small $d$};
\draw (E) node[right] {\small $e$};
\draw (Xd) node[below] {\small $x$};
\end{scope}
\begin{scope}[shift={(0,0.7)}]
\draw[thick, ->] (1.25,0) -- (2.25,0);
\draw (1.75,0) node[above] {\scriptsize $(0)$};
\draw[thick,->] (1.25+3.5,0) -- (2.25+3.5,0);
\draw (1.75+3.5, 0) node[above] {\scriptsize $(1A)$};
\draw[thick,->] (1.25+7.0, 0) -- (2.25+7.0,0);
\draw (1.75+7.0,0) node[above] {\scriptsize $(1B)$};
\end{scope}

\end{tikzpicture}
}
\caption{Illustration for Lemma \ref{lem:pachner}}
\label{fig:1A1B}
\end{figure}
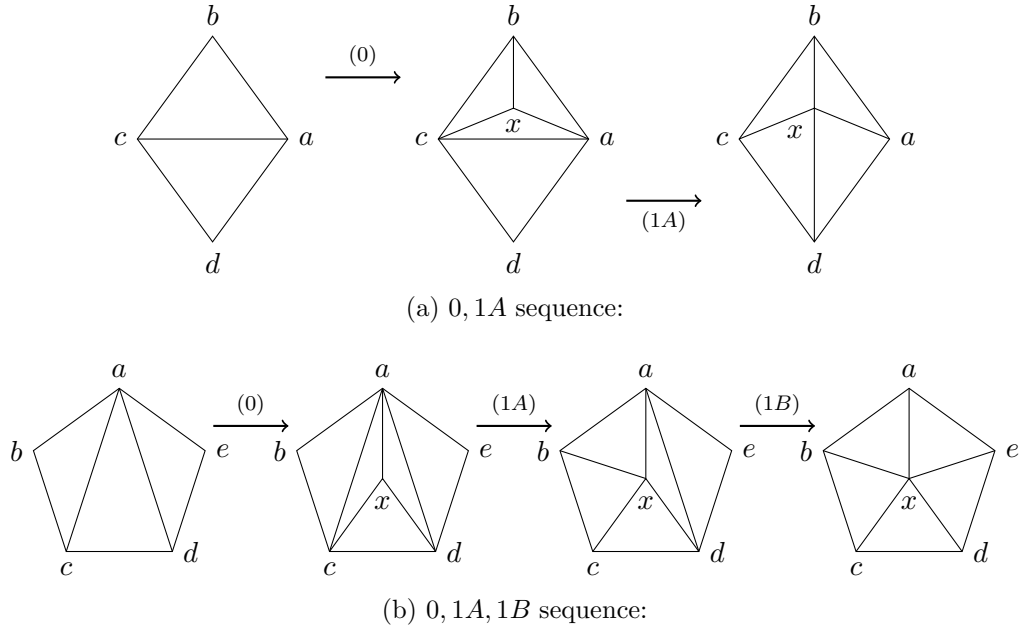

\noindent {\bf Case 1:} $X$ has a vertex of degree $3$. Let $x$ be the
vertex of degree $3$ in $X$. Let $a,b,c$ be the neighbours of $x$ in $X$.
As each edge is in exactly two triangles, it follows that the triangles $xab,xbc,xac$ are faces in $X$. However $abc$ cannot
be a face of $X$, since otherwise $X[\{a,b,c,d\}]\cong S^2_4$. Now consider $Y :=
X[V(X)\bs \{x\}]\cup \{abc\}$. It is easily seen that $X$ is obtained from
$Y$ by a $0$-move. The lemma now follows by the induction hypothesis. 

\medskip

\noindent {\bf Case 2:} $X$ has a vertex of degree $4$. Let $x$ be a vertex of degree $4$ and let $a,b,c,d$ be its neighbours such that
$\lk{X}{x}$ is the 4-cycle $C_4(a,b,c,d)$. Since $K_5$ is not
planar, we have $K_5\not\subseteq X$. Hence there is a pair of
non-adjacent vertices among $a,b,c,d$. Assume $ac$ is a non-edge. Define
$Y := X[V(X)\bs \{x\}]\cup \{abc,acd,ac\}$. Then $X$ can be obtained from
$Y$ by a $0$-move followed by a $1A$-move as illustrated in Figure
\ref{fig:1A1B}\subref{subfig:01A}.
The result then follows by invoking the induction hypothesis for $Y$.

\medskip

\noindent {\bf Case 3:} All vertices of $X$ have degree $5$ or more. Let $x$
be a vertex of degree $5$. Let $a,b,c,d,e$ be neighbours of $x$ such that
$\lk{X}{x}$ is the cycle $C_5(a,b,c,d,e)$. Since the induced subgraph on
the vertex set $\{x,a,b,c,d,e\}$ is planar, by Euler's bound (Lemma \ref{lem:graphs}(b)) it can have at most
$3\cdot 6-6=12$ edges. It follows then that at least one vertex among
$a,b,c,d,e$ has two non-neighbours. Let $a$ be such a vertex, with nonedges $ac$ and $ad$. Consider $Y := X[V(X)\bs \{x\}]\cup
\{abc,acd,ade,ac,ad\}$. Then $X$ can be obtained from $Y$ by a sequence of
a $0$-move, a $1A$-move and a $1B$-move as illustrated in Figure
\ref{fig:1A1B}\subref{subfig:01A1B}. The lemma
follows by using the induction hypothesis for $Y$.
\end{proof}

\subsection{The sigma-vector and mu-vector}

For any set $V$ and any integer $i\geq 0$, the collection of
all $i$-element subsets of $V$ will be denoted by $\binom{V}{i}$. Let $X$ be a simplicial complex of dimension $d$.
As usual, $\tilde{\beta}_i(X)$ denotes the reduced $i$th homology of $X$. Thus,
$\tilde{\beta}_0(X)=\beta_0(X; \mathbb{Z}_2)-1$ and $\tilde{\beta}_i(X)=\beta_i(X; \mathbb{Z}_2)$ for $i> 0$. We recall
the following definitions from \cite{BDEJC}.

\begin{definition}\label{def:sigmamu}
Let $X$ be a $d$-dimensional simplicial complex on $n$ vertices. The {\em
sigma-vector} $(\sigma_0,\sigma_1,\ldots,\sigma_d)$ of $X$ is defined by
\begin{align}\label{eq:sigma}
\sigma_i &=\sigma_i(X)=\sum_{A\subseteq
V(X)}\frac{\tilde{\beta}_i(X[A])}{\binom{n}{|A|}}, \quad
0\leq i\leq d, 
\end{align}
where $\tilde{\beta}_i = \beta_i$ if $i\neq 0$ and $\tilde{\beta}_0 = \beta_0-1$.  The {\em mu-vector} $(\mu_0,\mu_1,\ldots,\mu_d)$ of $X$ is defined by
\begin{align}\label{eq:mu}
\mu_0 &= \mu_0(X) = 1, \nonumber \\
\mu_i &= \mu_i(X) = \delta_{i1} + \frac{1}{n}\sum_{x\in
V(X)}\sigma_{i-1}(\lk{X}{x}), \quad 1\leq i\leq d.
\end{align}
Here $\delta_{i1}$ is the Kronecker delta, i.e., $\delta_{i1} = 0$ if $i\neq 1$ and $\delta_{11} = 1$.
\end{definition}

The following result follows from Theorem 2.6 in \cite{BDEJC}.

\begin{proposition}\label{prop:bdejc}
Let $X$ be a neighbourly simplicial complex of dimension $d$. Then
$\beta_i\leq \mu_i$ for all $0\leq i\leq d$. 
\end{proposition}

\section{Separation index of a graph}

For a graph $G$, let $q(G)$ denote the number of connected components of
$G$. We know that $\beta_0(X)$ is the number of connected components of
$X$, which is same as the number of connected components in the
$1$-skeleton of $X$. Thus, we see that $\sigma_0(X)$ is (roughly) a weighted average of
the number of connected components over all induced subgraphs of the
$1$-skeleton of $X$. This motivates us to define:

\begin{definition}\label{def:sepindex}
Let $G$ be a graph on $n$ vertices. We define the {\em separation index} of $G$ to be
$s(G)$, given by
\begin{align}\label{eq:sepindex}
s(G) & := \sum_{A\subseteq V(G)} \frac{q(G[A])-1}{\binom{n}{|A|}}= \sum_{i=0}^n s_i(G),
\end{align}
where 
\begin{align*} 
s_i(G) & := \frac{1}{\binom{n}{i}}\sum_{\stackrel{A\subseteq V(G)}{|A|=i}}
\big(q(G[A])-1\big). 
\end{align*}
So, $s_i(G)$ is one less than the average number of components in an induced subgraph with $i$ vertices. 
\end{definition}

It is easily seen that for any graph $G$, we have $-1\leq s(G)\leq (n+1)(n-2)/2$, with the lower and
upper bounds attained by the complete graph $K_n$ and its complement respectively. If $X$ is obtained from $Y$ by a bistellar move then the relation between the sigma-vectors of $X$ and $Y$ is given in \cite{BDEJC}. 
We present an elementary proof of the following. 

\begin{lemma}\label{lem:zeromove}
Let $X$ be obtained from $Y$ by a bistellar $0$-move. If $Y$ triangulates $S^2$ and $f_0(Y)=n$ then 
$s(X)=(n+2)s(Y)/(n+1)+(n+2)/20$.
\end{lemma}

\begin{proof}
Let $X$ be obtained from $Y$ by starring the vertex $v_0$ in the face $v_1v_2v_3$. Let $G=\skel{1}{Y}$ and
$H=\skel{1}{X}$. Then $H$ is obtained from $G$ by introducing a new
vertex $v_0$ and joining it to vertices $v_1,v_2,v_3$. We recall that
$v_1,v_2,v_3$ are mutually adjacent in $G$. Let $V := V(Y)$. Now
from \eqref{eq:sepindex} we have 
\begin{align}\label{eq:zeromove}
s(H) & = \sum_{S\subseteq V} \frac{q(H[S])-1}{\binom{n+1}{|S|}} +
\sum_{S\subseteq V} \frac{q(H[S\cup \{v_0\}])-1}{\binom{n+1}{|S|+1}}.
\end{align}
We note that $q(H[S])=q(G[S])$ for $S\subseteq V$, $q(H[S\cup \{v_0\}])=q(G[S])$ when $S$ has a non-empty intersection with 
$N(v_0)=\{v_1, v_2, v_3\}$ (which we denote by $S\meets N(v_0)$), and $q(H[S\cup \{v_0\}])=q(G[S])+1$ when
$S$ does not intersect $N(v_0)$ (which we denote by $S\nmeets N(v_0)$). We
split the second summation in \eqref{eq:zeromove} to get 

\begin{align*} 
s(H) &= \sum_{S\subseteq V} \frac{q(G[S])-1}{\binom{n+1}{|S|}} + \sum_{S\nmeets N(v_0)} \frac{q(G[S])}{\binom{n+1}{|S|+1}} + 
\sum_{S\meets N(v_0)}\frac{q(G[S])-1}{\binom{n+1}{|S|+1}}  \\
&= \sum_{S\subseteq V} \frac{q(G[S])-1}{\binom{n+1}{|S|}} + 
\sum_{S\subseteq V} \frac{q(G[S])-1}{\binom{n+1}{|S|+1}} + \sum_{S\nmeets N(v_0)}\frac{1}{\binom{n+1}{|S|+1}} \\
&= \sum_{S\subseteq V} \frac{q(G[S])-1}{\binom{n}{|S|}}\left[\frac{n-|S|+1}{n+1}+\frac{|S|+1}{n+1}\right] + 
\sum_{S\nmeets N(v_0)}\frac{1}{\binom{n+1}{|S|+1}}  \\
&= \frac{n+2}{n+1}\times s(G)  + \sum_{S\nmeets N(v_0)}\frac{1}{\binom{n+1}{|S|+1}} 
=  \frac{(n+2)s(G)}{n+1}  +\sum_{j=0}^{n-3}
\frac{\binom{n-3}{j}}{\binom{n+1}{j+1}} \\
&= \frac{(n+2)s(G)}{n+1}  +\frac{3!\times(n-3)!}{(n+1)!}\times \binom{n+2}{5} = \frac{(n+2)s(G)}{n+1}  + \frac{n+2}{20}.
\end{align*}
The last but one equality follows from the fact that $\sum_{j=0}^{n-3}
\binom{n-j}{3}\times(j+1)=\binom{n+2}{5}$. 
\end{proof}

The following corollary is a particular case of Theorem 3.5 (c) of \cite{BDEJC}. We have included here with an elementary proof for the sake of completeness. 

\begin{corollary}\label{cor:svector-stacked}
Let $X$ be a stacked $2$-sphere. If $f_0(X)=n$ then $s(X)=\frac{(n-8)(n+1)}{20}$. 
\end{corollary}

\begin{proof}
Since an $n$-vertex stacked 2-sphere is obtained from the standard 2-sphere $S^2_4$ by a sequence of $n-4$ bistellar 0-moves, the corollary follows by induction and Lemma \ref{lem:zeromove}.
\end{proof}

\begin{proof}[Proof of Theorem $\ref{thm:result1}$]
We have already shown that when $X$ is a stacked $2$-sphere on $n$
vertices, $s(X)=(n-8)(n+1)/20$. It remains to show that
$s(X)<(n-8)(n+1)/20$ when $X$ is not a stacked $2$-sphere. We proceed by
induction. For $n=4,5$ the result is vacuously true as all triangulated 
2-sphere on at most $5$ vertices are stacked. Assume that $n\geq 5$ and that
the result is true for all triangulated 2-spheres $Y$ with $f_0(Y)\leq n$. Let $X$ be a triangulated 2-sphere with $n+1$ vertices, which is not stacked. We consider the following cases:

\medskip

\noindent {\bf Case 1:} $X$ has a vertex of degree $3$. Then, from the proof of
Lemma \ref{lem:pachner}, there exists an $n$-vertex triangulated 2-sphere 
$Y$ such that $X$ is obtained from $Y$ by a bistellar
$0$-move. Since $X$ is not stacked, it follows that $Y$ is not stacked. Therefore, by the induction hypothesis $s(Y)<(n-8)(n+1)/20$. 
Then, by Lemma
\ref{lem:zeromove}, we have $s(X)= (n+2)s(Y)/(n+1)+(n+2)/{20} < (n-7)(n+2)/20$. 

\medskip

\noindent{\bf Case 2:} $X$ has a vertex of degree $4$, but no vertices of degree
 less than $4$. By the proof of Lemma \ref{lem:pachner}, there exist
vertices $a,b,c,d,x$ and triangulated 2-spheres $Y,Z$ such that: {\rm
(i)} $\lk{X}{x}=C_4(a,b,c,d)$, {\rm (ii)} $Y := X[V(X)\bs \{x\}]\cup
\{abc,adc,ac\}$ and {\rm (iii)} $Z$ is obtained from $Y$ by starring the vertex $x$ in the face $abc$, and $X$ is obtained from $Z$ by the 1-move $ac\mapsto dx$. 
By the induction hypothesis $s(Y) \leq (n-8)(n+1)/20$.

Let $U = V(X) \setminus \{a, b, c, d, x\}$. For $A\subseteq U$, let $\varphi(A\cup \{a, c\}) = A\cup\{d, x\}$. 
Let $T^+$ (reps., $T^-$) be the family of those subsets $S$ of $V(X)$ such that $q(X[S])>q(Z[S])$ (resp., $q(X[S]) <
q(Z[S])$). As removing (resp., adding) an edge
increases (resp., decreases) the number of components by at most one, we have $T^+ = \{S\subseteq V(X):
q(X[S])=q(Z[S])+1\}$, and $T^- = \{S\subseteq V(X):
q(X[S])=q(Z[S])-1\}$. For $S\in T^+$, since $q(X[S]) > q(Z[S])$, $f_1(X[S]) < f_1(Z[S])$. This implies, since the only edge in $Z$ which is not in $X$ is $ac$, $\{a, c\}\subseteq S$ for $S\in T^+$. Similarly $S\in T^- \Rightarrow \{d,x\}\subseteq S$. We have the
following cases. 

\smallskip

\noindent {\bf Subcase 2a:} $bd$ is not an edge in $Z$. Let $A\cup
\{a,c\}$ be a set in $T^+$. Then $A$ does not contain any common
neighbours of $a$ and $c$ (otherwise $q(X[A\cup \{a,c\}]) = q(Z[A\cup \{a,c\}])$). Thus, $b,d,x\not\in A$ and hence 
$A\subseteq U$. Then $\varphi(A\cup \{a,c\})=A\cup \{d,x\}$. Since
$Z[A\cup \{d,x\}]$ does not have a $d$-$x$-path and $dx$ is an edge in
$X$, we have $q(X[A\cup \{d,x\}]) = q(Z[A\cup \{d,x\}])-1$. Therefore,
$\varphi$ is an injection from $T^+$ to $T^-$, which preserves sizes of
sets. Then, it follows that $s(X)\leq s(Z)$. Consider the set $S=\{b,d,x\}$.
Since $bd$ is not an edge, we have $S\in T^-$. Since $S\not\in
\varphi(T^+)$, by Lemma \ref{lem:zeromove}, we have  $s(X) < s(Z) = \frac{(n+2)s(Y)}{n+1}+\frac{n+2}{20} \leq 
\frac{(n+2)(n-8)}{20} + \frac{n+2}{20} = \frac{(n+2)(n-7)}{20}$. 

\smallskip

\noindent{\bf Subcase 2b:} $bd$ is an edge in $Z$. As in Subcase 2a, we can show that
$s(X)\leq s(Z)$. Suppose $Y$ is a stacked sphere. Then $Y$ has at least two vertices of
degree $3$. Since $d_X(u)=d_Y(u)$ for all $u\in U$, we conclude that at least
two vertices among $a,b,c,d$ have degree $3$ in $Y$. 
Since $K_4\subseteq Y[\{a,b,c,d\}]$, this implies that $Y[\{a,b,c,d\}]\cong S^2_4$. Thus, $S^2_4 \subseteq Y$. 
This is not possible since $Y$ is a triangulated 2-sphere and $f_0(Y)\geq 5$. 
Therefore, $Y$ is not a stacked sphere and hence by the induction hypothesis for $Y$, $s(Y) < (n-8)(n+1)/20$. 
Then, by Lemma \ref{lem:zeromove}, $s(X) \leq s(Z) = \frac{(n+2)s(Y)}{n+1}+\frac{n+2}{20} <  
\frac{(n+2)(n-8)}{20} + \frac{n+2}{20} = \frac{(n+2)(n-7)}{20}$. 

\medskip


\noindent{\bf Case 3:} All vertices in $X$ have degree $5$ or more. By 
the proof of Lemma \ref{lem:pachner} there exist vertices $a,b,c,d,e$ and
$x$ in $X$ and triangulated 2-spheres $Y,Z$ such {\rm (i)} $\lk{X}{x}=C_5(a,b,c,d,e)$, {\rm (ii)} $Y:=X[V(X)\bs
\{x\}]\cup \{abc,acd,ade,ac,ad\}$, $Z$ is obtained from $Y$ by starring the vertex $x$ in 
the face $acd$, and {\rm (iii)} $X$ is
obtained from $Z$ by a $1A$-move ($ac\mapsto bx$) followed by a 
$1B$-move ($ad\mapsto ex$). 
Observe that for any vertex $u\in
V(Y)$, $d_Y(u)\geq d_X(u)-1$. Thus $d_Y(u)\geq 4$ for all $u\in V(Y)$.
Therefore $Y$ is not a stacked sphere. Hence by the induction hypothesis, $s(Y)
< (n-8)(n+1)/20$. 

We now show that $s(X)\leq s(Z)$. For $A\subseteq W := V(X)\bs\{a,b,c,d,e,x\}$, let 
\begin{align*}
\varphi(A\cup \{a,c,d\})  = A\cup \{b,e,x\}, & ~~ 
\varphi(A\cup \{a,c\})  = A\cup \{b,x\}, \\
\varphi(A\cup \{a,c,e\})  = A \cup \{c,e,x\}, & ~~ 
\varphi(A\cup \{a,d\})  = A \cup \{e,x\}, \mbox{ and}\\
\varphi(A\cup \{a,b,d\})  = A \cup \{b,d,x\}. &
\end{align*}

Let $V^{+} := \{S\subseteq V(X) \, : \, q(X[S]) > q(Z[S])\}$ and $V^{-} := \{S\subseteq V(X) \, : \, q(X[S]) < q(Z[S])\}$. 
Observe that for a set $S\in V^{+}$, $q(X[S]) = q(Z[S]) +1$ while for $S\in V^{-}$, $q(X[S]) = q(Z[S]) -1$ or $q(Z[S])-2$. 

First we show that every set in $V^+$ occurs on the left side of the mapping $\varphi$. Since the
only edges of $Z$ not present in $X$ are $ac$ and $ad$, any set
$S\in V^+$ contains $a$ and at least one of $c$, $d$. Let $S\in V^+$. Then if
$\{a,c,d\}\subseteq S$, we see that $b,e,x\not\in S$ (otherwise the induced
subgraph on $S$ does not decompose further on removal of edges $ac$ and
$ad$). Thus $S=A \cup \{a,c,d\}$ for some $A \subseteq W$. Next
suppose $\{a,c\}\subseteq S$ but $d\not\in S$. Then $S$ does not contain an
$a$-$c$-path in $X$, therefore $b,x\not\in S$. Thus, we have two
possibilities, $S= A\cup \{a,c\}$ or $S=A\cup \{a,c,e\}$ for some
$A\subseteq W$. The case $\{a,d\}\subseteq S$ but
$c\not\in S$ is symmetric, and leads to the last two descriptions of the
sets in the mapping $\varphi$ above. Hence we have shown that $\varphi$ is defined on
$V^+$. 

It is easily checked that $\varphi$ is an injection that preserves
sizes of sets. To establish $s(X)\leq s(Z)$ it is enough to show that
$\varphi(S)\in V^-$ whenever $S\in V^+$. We argue each case
separately.

\smallskip

\noindent {\bf Subcase 3a:} $S=A\cup \{a,c,d\}$. Then $\varphi(S)=A\cup
\{b,e,x\}$. Observe that there is no $x$-$b$ or $x$-$e$-path in
$Z[A\cup \{b,e,x\}]$. But $b,e,x$ lie in the same connected component in $X$.
Thus $q(X[\varphi(S)]) < q(Z[\varphi(S)])$, and $\varphi(S)\in
V^-$.

\smallskip 

\noindent {\bf Subcase 3b:} $S=A\cup \{a,c\}$. In this case $\varphi(S)=A\cup
\{b,x\}$. Since there is no $b$-$x$-path in $Z[\varphi(S)]$ and $bx$
is an edge in $X$, we have $\varphi(S)\in V^-$. 

\smallskip

\noindent {\bf Subcase 3c:} $S=A\cup \{a,c,e\}$. Since $S\in V^+$, we conclude
there is no $c$-$e$-path in $X[A]$. Then $\{x,c\}$ are separated from
$e$ in $Z$ but are in the same component of $X$. Therefore,
$\varphi(S)=A\cup \{c,e,x\}\in V^-$. 

\smallskip
 
The other two cases, namely, $S= A\cup \{a, d\}$ and $S= A\cup \{a, b, d\}$ are symmetric to  Subcases 3b and 3c respectively. 
Thus, we have $s(X)\leq s(Z)$. 

Since $Z$ is obtained from $Y$ by a $0$-move, by Lemma
\ref{lem:zeromove}, we have $s(Z)=(n+2)s(Y)/(n+1)+(n+2)/20$. Then, by the same argument as in Subcase 2b, $s(X) < (n+2)(n-7)/20$. This proves the theorem.
\end{proof}

\begin{corollary}\label{lem:sigma1}
 A neighbourly $n$-vertex triangulated $3$-manifold $X$ is in the class $\Kd{3}$ if and only if $\mu_1(X) = {(n-4)(n-5)}/{20}$. 
\end{corollary}

\begin{proof}
Since $X$ is an $n$-vertex neighbourly $3$-manifold each vertex-link is an $(n-1)$-vertex triangulated 2-sphere. From Theorem
\ref{thm:result1}, we have
\begin{align} \label{eq:last}
\mu_1(X) & = 1+\frac{1}{n}\sum_{x\in V(X)} s(\lk{X}{x}) 
\leq 1+\frac{1}{n}\sum_{x\in V(X)} \frac{n(n-9)}{20} = \frac{(n-4)(n-5)}{20}.
\end{align}
Now, if $\mu_1(X) = {(n-4)(n-5)}/{20}$ then (by \eqref{eq:last}) all the vertex-links in $X$ satisfy Theorem
\ref{thm:result1} with equality, and hence are stacked spheres. So, $X\in \Kd{3}$. Conversely, if $X \in \Kd{3}$ then all vertex-links are stacked 2-spheres. Therefore, by Theorem \ref{thm:result1} and \eqref{eq:last}, we get $\mu_1(X) = {(n-4)(n-5)}/{20}$. This proves the result. 
\end{proof}

\begin{proof}[Proof of Theorem $\ref{thm:result2}$]
Let $X$ be a tight-neighbourly $3$-manifold with $f_0(X)=n$. Then from \eqref{eq:tn}, we
have $\beta_1(X)=(n-4)(n-5)/20$. Also, from Proposition \ref{prop:novikswartz}\,(a), $X$ is neighbourly. 
Then, by \eqref{eq:last}, $(n-4)(n-5)/20 = \beta_1(X) \leq \mu_1(X) \leq (n-4)(n-5)/20$.
Therefore, $\mu_1(X) = (n-4)(n-5)/20$. The result now follows from Corollary \ref{lem:sigma1}. 
\end{proof}

\begin{proof}[Proof of Corollary $\ref{cor:tntight}$]
Follows from Theorem \ref{thm:result2} and Proposition \ref{prop:bd}. 
\end{proof}

\begin{remark} \label{remark:converse}
We can give explicit examples of neighbourly members of $\Kd{3}$ which are not 
tight-neighbourly, thus disproving the converse of Theorem \ref{thm:result2}. 
By Perles' result (see \cite{BD2k+1}), a polytopal neighbourly 3-sphere is in 
$\Kd{3}$. In particular, the boundary complex of the cyclic 4-polytope is a 
neighbourly member of $\Kd{3}$. Therefore, by Corollary \ref{lem:sigma1}, if $S$ 
is any polytopal neighbourly 3-sphere with $f_0(S) \geq 6$ then $S$ is a 
neighbourly member of $\Kd{3}$ and 
$\beta_1(S) = 0 < (f_0(S)-4)(f_0(S)-5)/20 = \mu_1(S)$.
\end{remark}

\begin{proof}[Proof of Theorem~\ref{thm:two}]  
For $n=4$ the statement is true since there is only one $4$-vertex triangulated $2$-sphere which is both flag and stacked. So, assume that $n\geq 6$. 

In order to prove the theorem we consider an arbitrary non-flag 
$n$-vertex triangulated $2$-sphere $T$ and construct an $n$-vertex triangulated $2$-sphere $S$ such that $s(S) < s(T )$. 

Let $T$ be a non-flag $n$-vertex triangulated 2-sphere and $n\geq 6$.  
Then, there exists an induced 
$3$-cycle in $T$, i.e., vertices $u, v, w \in V(T)$ such that 
$T[\{u, v, w\}] = \{uv, uw, vw \}$. Since $T$ is a triangulated 2-sphere 
$T[\{u, v, w\}]$ divides $T$ into two 2-discs $D_1$ and $D_2$. Let $V_1$ and 
$V_2$ be the sets of interior vertices of $D_1$ and $D_2$. In particular, 
$V(T) =\{u, v, w\} \sqcup V_1\sqcup V_2$, and 
$D_1 := T[\{u, v, w\} \sqcup V_1]$, 
$D_2 := T[\{u, v, w\} \sqcup V_2]$. 

Since $uvw$ is not a face, both $V_1$ and 
$V_2$ are non-empty. Since $n\geq 6$, we can assume without loss that 
$\#(V_1) > 1$. 
Let the boundary triangles of $D_1$ be $uvb$, $vwc$ and $uwd$. Since
$\#(V_1) > 1$, we cannot have $b=c=d$. Without loss of generality, assume that $b\neq c$. Since 
$D_1$ is a 2-disc, $bw$ and $cu$ cannot both be edges. Again, without loss, 
assume that $bw$ is a non-edge in $T$. Let $uva$ be the 2-face in $D_2$ 
containing $uv$ (i.e., $uva$ and $uvb$ are the 2-faces in $T$ containing $uv$). 
Since $T[\{u, v, w\}]$ is separating in $T$ and $a, b$ are in different 
components, it follows that $ab$ is not an edge of $T$. 

\smallskip

Let $S$ be the triangulated $2$-sphere obtained from $T$ by the 
edge-flip $uv \mapsto ab$. (That is,  
$S= (T \setminus \{uv, uva, uvb\})\cup \{ab, abu, abv\}$.)
We prove $s(S) < s(T)$.

Let $A\subseteq V(T) = V(S)$. 

\medskip

\noindent {\bf Case 1:} $w\in A$. Since the only edge of $T$ which is 
not an edge in $S$ is $uv$, both $uw$ and $vw$ are edges in 
$S$ and $T$. The only way that $q(S[A]) > q(T[A])$ is if the deletion of
$uv$ generates an additional connected component. For this, both $u$ and $v$
must be in $A$ and must be in distinct connected components in $S[A]$. However,
$u$ and $v$ will always be joined in $S[A]$ via $w$ and hence 
$q(S[A]) \leq q(T[A])$.

\medskip

\noindent {\bf Case 2:} $w\not\in A$. Here we have the following cases. 

\smallskip

\noindent {\bf Subcase 2a:} One of $a$ or $b$ is in $A$. In this case, by the similar arguments as in Case 1, $q(S[A]) \leq q(T[A])$. 

\smallskip

\noindent {\bf Subcase 2b:} Both $a,b \not \in A$ and one of $u$ or $v$ is not in $A$. Then, by similar arguments,  $q(S[A]) \leq q(T[A])$. 

\smallskip

\noindent {\bf Subcase 2c:} $a, b\not\in A$ and $u, v\in A$. 
In this case, we have $q(S[A]) = q(T[A])$ or $q(S[A]) = q(T[A]) +1$. Now, 
consider the set $A^{\prime} := (A \setminus \{u, v\}) \cup \{a, b\}$. 
Since $T[\{u, v, w\}]$ is separating in $T$, $a$ and $b$ are in distinct
components, $u, v, w\not\in A^{\prime}$, $a, b\in A^{\prime}$, $ab\not\in 
T$, $ab\in S$, it follows that $q(S[A^{\prime}]) = q(T[A^{\prime}])-1$. Thus, $q(S[A]) + q(S[A^{\prime}]) \leq q(T[A])+ q(T[A^{\prime}])$ and 
$\#(A) = \#(A^{\prime})$. 

Since there is a one-to-one correspondence between the sets of $A$ and  
set of $A^{\prime}$ and $\#(A) = \#(A^{\prime})$, it follows that 
$s(S) \leq s(T)$. 

Note that, for $A= \{a, b, w\}$ 
we have $q(S[A]) = q(T[A])-1$ since $bw$ is not an edge of $T$ or 
$S$ and it 
follows that $s(S) < s(T)$. This completes the proof. 
\end{proof}

\noindent {\bf Acknowledgements:} This work is supported by the Department
of Industry and Science, Australia and the 
Department of Science \& Technology, India, under the Australia-India Strategic Research Fund (project AISRF06660).
The first author is also supported by the Australian Research Council (project DP1094516). 
The third author thanks the National Mathematics Initiative, India for support.
The second and third authors are also supported by the UGC Centre for Advanced Studies. 
This work was done when the second and third authors were visiting the University of 
Queensland (UQ), and they thank UQ for their hospitality. 
The authors thank the anonymous referees for many useful comments.


\end{document}